\documentclass{amsart}
\usepackage{amssymb}
\usepackage{amsthm}

\usepackage{amssymb,amsmath,amsthm,amsxtra,amsfonts}
\newcommand{\bdone}{{\boldsymbol{1}}}
\newcommand{\bdnot}{{\boldsymbol{0}}}

\newcommand{\tr}{\text{\rm{Tr}}}

\newcommand{\beq}{\begin{equation}}
\newcommand{\eeq}{\end{equation}}
\newcommand{\ba}{\begin{align*}}
\newcommand{\ea}{\end{align*}}

\newcommand{\norm}[1]{\lVert#1\rVert}

\DeclareMathOperator{\real}{Re}

\allowdisplaybreaks
\numberwithin{equation}{section}

\newtheorem{theorem}{Theorem}
\newtheorem{lemma}{Lemma}

\newtheorem*{definition}{Definition}
\theoremstyle{remark}
\newtheorem*{remarks}{Remarks}

\theoremstyle{remark}
\newtheorem{example}{Example}

\begin{document}


\title{Equivalence classes of block Jacobi matrices}


\author{Rostyslav Kozhan}
\email{rostysla@caltech.edu}
\address{California Institute of Technology\\
Department of Mathematics 253-37\\
Pasadena, CA 91125, USA}
\subjclass[2000]{Primary }
\date{\today}
\commby{}

\begin{abstract}
The paper contains two results on the equivalence classes of block Jacobi matrices: first, that the Jacobi matrix of type $2$ in the Nevai class has $A_n$ coefficients converging to $\bdone$, and second, that under an $L^1$-type condition on the Jacobi coefficients, equivalent Jacobi matrices of type $1$, $2$ and $3$ are pairwise asymptotic.
\end{abstract}
\maketitle

%
%
%

\begin{section}{Introduction and results}

A block Jacobi matrix is an infinite  matrix of the form
\begin{equation*}
J=\left(
\begin{array}{cccc}
B_1&A_1&\mathbf{0}&\\
A_1^* &B_2&A_2&\ddots\\
\mathbf{0}&A_2^* &B_3&\ddots\\
 &\ddots&\ddots&\ddots\end{array}\right),
\end{equation*}
where $A_n, B_n$ are $l\times l$ matrices with $A_n$ invertible. The sequences $A_n$ and $B_n$ are called Jacobi parameters of $J$.

Two block Jacobi matrices $J$ and $\widetilde{J}$ are called equivalent if their Jacobi parameters satisfy
\begin{equation}\label{eq2.1}
\widetilde{A}_n=\sigma_n^*  A_n \sigma_{n+1}, \quad
\widetilde{B}_n=\sigma_n^*  B_n \sigma_n
\end{equation}
for unitary $\sigma_n$'s with $\sigma_1=\bdone$. The definition comes from the fact that~\eqref{eq2.1} holds if and only if the (matrix-valued) spectral measures of $J$ and $\widetilde{J}$ coincide (see \cite{DPS} for the details).

Using the convention $p_{-1}=\bdnot$, $A_0=\bdone$, $p_0=\bdone$, the recurrence
\begin{equation}\label{recur}
x p_n(x)=p_{n+1}(x)A_{n+1}^* +p_n(x)B_{n+1}+p_{n-1}(x) A_n, \quad n=0,1,\ldots,
\end{equation}
allows one to define a sequence of matrix-valued polynomials, which 
turn out to be (right-) orthonormal with respect to the above mentioned spectral measure.

Inductively it is easy to see that
\begin{equation}\label{eq2.2}
\widetilde{p}_n(x)=p_n(x)\sigma_{n+1},
\end{equation}
where $\widetilde{p}_n$ are the orthonormal polynomials for $\widetilde{J}$.

We say that a block Jacobi matrix is of type $1$ if $A_n>0$ for all $n$,
of type $2$ if $A_1A_2\ldots A_n>0$ for all $n$, and of type $3$ if every $A_n$ is lower triangular with strictly positive elements on the diagonal. Each equivalence class of block Jacobi matrices contains exactly one matrix of type $1$, $2$ and $3$ (follows from the uniqueness of the polar and QR decompositions, see \cite{DPS} for the proof).

%
%
%

We say that $J$ is in the Nevai class if
\begin{equation*}
B_n\to\bdnot, \quad A_nA_n^* \to \bdone.
\end{equation*}
Note that this definition is invariant within the equivalence class of Jacobi matrices.

\begin{theorem}\label{thm1}
Assume $J$ belongs to the Nevai class. If $J$ is of type $1$, $2$ or $3$, then $A_n\to\bdone$ as $n\to\infty$.
\end{theorem}

This result was proven in \cite{DPS} for the type $1$ and $3$ cases, and was left open for type $2$. It is proven here in Section $2$.

Note that the essence of Theorem~\ref{thm1} is to show that $\sigma_n^* \sigma_{n+1}\to\bdone$, where $\sigma_n$'s are the unitary coefficients from~\eqref{eq2.1} for $J$, $\widetilde{J}$ of type $1$, $2$ or $3$. Looking at~\eqref{eq2.2}, it is clear that any result on the asymptotics of $p_n$ (see e.g.~\cite{AN},~\cite{Kozhan},~\cite{Kozhan2}) would involve the limit $\lim_{n\to\infty} \sigma_n$. This explains the  need for the following definition.

\begin{definition}
Two equivalent matrices $J$ and $\widetilde{J}$ with~\eqref{eq2.1} are called asymptotic to each other if the limit $\lim_{n\to\infty} \sigma_n$ exists.
\end{definition}

Clearly this is an equivalence relation on the set of equivalent block Jacobi matrices. Thus, establishing Szeg\H{o} asymptotics (which simply means $\lim_{n\to\infty} z^n p_n(z+z^{-1})$ exists) for any block Jacobi matrix immediately implies the corresponding asymptotics for any of the Jacobi matrices asymptotic to the original one.

\begin{theorem}\label{thm2}
Assume
\begin{equation}\label{eq2.6}
\sum_{n=1}^\infty \left[ \norm{\bdone-A_n A_n^*}+ \norm{B_n}\right]<\infty.
\end{equation}
Then the corresponding Jacobi matrices of type $1$, $2$ and $3$ are pairwise asymptotic.
\end{theorem}
\begin{remarks}
1. The condition~\eqref{eq2.6} doesn't depend on the choice of the representative of the equivalence class of equivalent matrices.
\smallskip

2. The proof also shows that any Jacobi matrix, for which eventually each $A_n$ has real eigenvalues, is also asymptotic to type $1$, $2$, $3$.

\smallskip

3. An example of an equivalence class of block Jacobi matrices that fails~\eqref{eq2.6} and that has type $1$ and type $2$ nonasymptotic to each other can be found at the end of Section $2$.
\end{remarks}

\bigskip

\noindent\textbf{Acknowledgements.} The author would like to thank Prof Barry Simon for helpful comments.

\end{section}
\begin{section}{Proofs of the results}

We will be using the following lemma from \cite{Li}. For self-containment purposes we give a proof of it in the Appendix.

\begin{lemma}[Li \cite{Li}]\label{li} Let $\phi$ be the map that takes any invertible matrix $T$ to the unitary factor $U$ in its polar decomposition $T=|T|U$, where $|T|=\sqrt{TT^*}$. Then for any invertible $l\times l$ matrices $B, D$ the following holds
\begin{equation*}
||\phi(B)-\phi(BD)||_{HS}\le \sqrt{||\bdone-D^{-1}||_{HS}^2+||\bdone-D||_{HS}^2},
\end{equation*}
where $||\cdot||_{HS}$ is the Hilbert--Schmidt norm.
\end{lemma}


\medskip

\begin{proof}[Proof of Theorem~\ref{thm1}]
For type $1$ and $3$, the statement is proven in~\cite{DPS}.

Assume $J$ is of type $2$. Denote by $\widehat{J}$ the type $1$ Jacobi matrix equivalent to $J$. Denote its Jacobi parameters by $\widehat{A}_n, \widehat{B}_n$, and let
\begin{equation}\label{eq2.5}
{A}_n=\sigma_n^*  \widehat{A}_n \sigma_{n+1}
\end{equation}
for some unitaries $\sigma_n$. Since $\widehat{A}_n\to\bdone$, we get ${A}_n=\sigma_n^*  \widehat{A}_n \sigma_{n+1}=\left(\sigma_n^*  \widehat{A}_n \sigma_{n}\right)\sigma_n^* \sigma_{n+1}$ converges to $\bdone$ if and only if $\lim_{n\to\infty} \sigma_n^* \sigma_{n+1}=\bdone$.


Denote $Q_n=A_1\ldots A_n$, which is a positive-definite matrix. Note that $\widehat{Q}_n=\widehat{A}_1\ldots \widehat{A}_n=A_1\ldots A_n\sigma_{n+1}^*=Q_n\sigma_{n+1}^*$, so $Q_n=|\widehat{Q}_n|$ and $\sigma_{n+1}=\phi(\widehat{Q}_n)^*$. Here $\phi$ is the same as in Lemma~\ref{li}.

Now, $\widehat{A}_{n+1}\to\bdone$ together with Lemma~\ref{li} implies that $\phi(\widehat{Q}_{n+1})-\phi(\widehat{Q}_n)=\phi(\widehat{Q}_n\widehat{A}_{n+1})-\phi(\widehat{Q}_{n})\to \bdnot$. Thus, $\sigma_{n+1}-\sigma_n\to\bdnot$, and $\lim_{n\to\infty} \sigma_n^* \sigma_{n+1}=\bdone$.
\end{proof}

For the type $3$ case of Theorem \ref{thm2}, we will need the following lemma. Recall that the singular values of a matrix $A$ are defined to be the eigenvalues of $|A|$.

\begin{lemma}\label{lemma} There exists a constant $c$ such that for all $l\times l$ matrices $A$
\begin{equation}\label{ineq}
\sum_{j=1}^l (\sigma_j-|\lambda_j|) \le c\sum_{j=1}^l (1-\sigma_j)^2,
\end{equation}
where $\{\lambda_j\}_{j=1}^l$ and $\{\sigma_j\}_{j=1}^l$ are the eigenvalues and singular values of $A$, ordered by $|\lambda_1|\ge\ldots \ge |\lambda_l|$, $\sigma_1\ge\ldots\ge \sigma_l\ge0$, where $c$ depends on $l$ only.
\end{lemma}
\begin{proof}
For  sufficiently large matrices $A$ the inequality is clear. It also holds for any compact set on which the right-hand side of \eqref{ineq} does not vanish. Therefore, we only need to worry about neighborhoods of matrices with $\sum_{j=1}^l (1-\sigma_j)^2=0$, that is, unitary matrices.

Consider any matrix $A$ within distance $1/2$ from the unitary group. Let $U=\phi(A)$ be the unitary factor in the polar decomposition of $A$. Since $\phi(A)$ is always the closest unitary to $A$ (see e.g. ~\cite{Bhatia}), we get
\begin{equation*}
\norm{A-U}\le 1/2 \quad \mbox{and} \quad \norm{|A|-\bdone}\le 1/2.
\end{equation*}
The second inequality immediately gives
$|\sigma_j-1|\le 1/2$, which in turn implies $\left| |\lambda_j|-1\right|\le 1/2$ by \eqref{e2.5} below.
The following basic facts are well-known (see e.g.~\cite{Weyl}):
\begin{gather}
\sigma_1\ge |\lambda_j| \ge \sigma_l \mbox{ for any }j ;\label{e2.5} \\
\prod_{j=1}^l |\lambda_j| = \prod_{j=1}^l \sigma_j .\label{e2.6}
\end{gather}
Let $\varepsilon_j=\sigma_j-1$, $\delta_j=|\lambda_j|-1$. Then from~\eqref{e2.6},
$$\delta_l=\frac{\prod_{j=1}^l \sigma_j}{\prod_{j=1}^{l-1} |\lambda_j|}-1=\frac{\prod_{j=1}^l (1+\varepsilon_j)-\prod_{j=1}^{l-1} (1+\delta_j)}{\prod_{j=1}^{l-1} (1+\delta_j)},$$
and so
\begin{equation}\label{eq.2.10}
\begin{aligned}
\sum_{j=1}^l &(\sigma_j-|\lambda_j|) = \sum_{j=1}^l (\varepsilon_j-\delta_j) \\
&= \frac{\prod\limits_{j=1}^{l-1} (1+\delta_j) \sum\limits_{j=1}^l \varepsilon_j-\prod\limits_{j=1}^{l-1} (1+\delta_j) \sum\limits_{j=1}^{l-1} \delta_j
-\prod\limits_{j=1}^l (1+\varepsilon_j)+\prod\limits_{j=1}^{l-1} (1+\delta_j)}{\prod\limits_{j=1}^{l-1} (1+\delta_j)}.
\end{aligned}
\end{equation}
The first-order terms (i.e. those involving only one of $\varepsilon$'s or $\delta$'s) of the numerator cancel out:
\begin{equation*}
 \sum_{j=1}^l \varepsilon_j-\sum_{j=1}^{l-1} \delta_j
-\left(1+\sum_{j=1}^l \varepsilon_j\right)+\left(1+\sum_{j=1}^{l-1} \delta_j\right)=0.
\end{equation*}
Now note that by \eqref{e2.5}, $|\delta_j|\le |\varepsilon_1|+|\varepsilon_l|$. Using this and $|\varepsilon_j \varepsilon_k |\le (\varepsilon_j^2+\varepsilon_k^2)/2$ we can bound
all of the second-order terms (i.e. those with $\varepsilon_j\varepsilon_k$, $\varepsilon_j\delta_k$ and $\delta_j\delta_k$) by $\widetilde{c} \sum_{j=1}^l \varepsilon_j^2$, where $\widetilde{c}$ will depend on $l$ only. All of the higher-order terms can be taken care of by using $|\varepsilon_j|<1, |\delta_j|<1$ to reduce it to the second-order. Finally, the denominator of the right-hand side of~\eqref{eq.2.10} is bounded below by $1/2^l$. Therefore, we obtain
\begin{equation*}
\sum_{j=1}^l (\sigma_j-|\lambda_j|) \le c \sum_{j=1}^l \varepsilon_j^2 = c \sum_{j=1}^l (1-\sigma_j)^2,
\end{equation*}
which proves our lemma.
\end{proof}

\begin{lemma}\label{lm3} There exists a constant $c$ so that
\begin{equation}\label{eq2.8}
\norm{\bdone-A}\le c  \norm{\bdone-|A|}
\end{equation}
for any $l\times l$ matrix $A$ with real positive eigenvalues, where $c$ depends on $l$ only.
\end{lemma}
\begin{proof}

By the equivalence of norms, we can prove \eqref{eq2.8} for the Hilbert--Schmidt norm instead. Let $\lambda_1\ge\ldots \ge \lambda_l>0$ be the eigenvalues of $A$, and let $\sigma_1\ge\ldots\ge \sigma_l>0$ be the singular values of $A$. Note
\begin{equation*}
\begin{aligned}
\norm{\bdone-A}_{HS}^2=\tr\left[(\bdone-A)(\bdone-A)^*\right]&=l-2\sum_{j=1}^l \real\lambda_j+\sum_{j=1}^l \tr AA^* \\
&= l-2\sum_{j=1}^l \lambda_j+\sum_{j=1}^l\sigma_j^2,\\
\norm{\bdone-|A|}_{HS}^2=\tr\left[(\bdone-|A|)^2\right]&=l-2\sum_{j=1}^l \sigma_j+\sum_{j=1}^l \sigma_j^2,
\end{aligned}
\end{equation*}
and so $\norm{\bdone-A}_{HS}^2 \le M \norm{\bdone-|A|}_{HS}^2$ holds if and only if
\begin{equation*}\label{eq.2.7}
2\sum_{j=1}^l (\sigma_j-\lambda_j) \le (M-1) \sum_{j=1}^l (1-\sigma_j)^2.
\end{equation*}
Since $\lambda_j=|\lambda_j|$, the previous lemma proves the result.
\end{proof}

\begin{proof}[Proof of Theorem~\ref{thm2}]
As in Theorem \ref{thm1}, let $\widehat{A}_n$ be of type $1$, and $A_n$ of type $2$ with the equivalence~\eqref{eq2.5}. Then keeping the notation of Theorem \ref{thm1} and using Lemma~\ref{li}, we have
\begin{multline*}\label{eq2.7}
\begin{aligned}
\sum_{n=1}^\infty \norm{\sigma_n&-\sigma_{n+1}}_{HS}= \sum_{n=1}^\infty \norm{\phi(\widehat{Q}_n)-\phi(\widehat{Q}_{n+1})}_{HS} \\
&\le \sum_{n=1}^\infty \sqrt{||\bdone-{\widehat{A}_{n+1}}^{-1}||_{HS}^2+||\bdone-\widehat{A}_{n+1}||_{HS}^2} \\
&\le \sum_{n=1}^\infty ||\bdone-{\widehat{A}_{n+1}}^{-1}||_{HS}+\sum_{n=1}^\infty ||\bdone-\widehat{A}_{n+1}||_{HS} \\
&\le (\sup_{n}||\widehat{A}_n||_{HS}+1) \sum_{n=1}^\infty ||\bdone-\widehat{A}_{n+1}||_{HS} \\
&\le (\sup_{n}||\widehat{A}_n||_{HS}+1) \sup_n{||(\bdone+\widehat{A}_n)^{-1}||_{HS}} \sum_{n=1}^\infty  ||\bdone-\widehat{A}_{n+1}^2||_{HS} <\infty,
\end{aligned}
\end{multline*}
since $\widehat{A}_n\to\bdone$ and so $\sup_{n}||\widehat{A}_n||_{HS}<\infty$, $\sup_n{||(\bdone+\widehat{A}_n)^{-1}||_{HS}}<\infty$.

This implies that $\sigma_n$ is Cauchy, and so converges.

\smallskip

An alternative indirect way of proving that type $1$ and type $2$ are asymptotic to each other  is as follows: it is proven in \cite{Kozhan} that under condition \eqref{eq2.6} Szeg\H{o} asymptotics  for the type $2$ block Jacobi matrix holds. In \cite{Kozhan2} the same fact is obtained for the type $1$ Jacobi matrix. Therefore~\eqref{eq2.2} implies that the limit $\lim_{n\to\infty} \sigma_n$ exists.

\medskip

Now assume that $\widehat{A}_n$ is of type $1$, and $A_n$ of type $3$ with the equivalence~\eqref{eq2.5}. Since all eigenvalues of $A_n$ are real and positive, Lemma~\ref{lm3} applies, and we get
\begin{equation*}
\sum_{n=1}^\infty \norm{\bdone-A_n}\le  c \sum_{n=1}^\infty \norm{\bdone-|A_n| } = c \sum_{n=1}^\infty \norm{\bdone-\widehat{A}_n }
\end{equation*}
since $|A_n|=\sigma_n^* \widehat{A}_n \sigma_n$ by~\eqref{eq2.5}.
Now $\sum_{n=1}^\infty ||\bdone-\widehat{A}_n||\le \sup_n ||(\bdone+\widehat{A}_n)^{-1}||\sum_{n=1}^\infty \norm{\bdone-\widehat{A}_n^2} <\infty$ which implies
\begin{multline*}
\begin{aligned}
\sum_{n=1}^\infty \norm{\sigma_n-\sigma_{n+1}} = \sum_{n=1}^\infty \norm{\bdone-\sigma_n^*\sigma_{n+1}} &\le \sum_{n=1}^\infty \norm{\bdone-A_n} + \sum_{n=1}^\infty \norm{A_n-\sigma_n^*\sigma_{n+1}} \\
&= \sum_{n=1}^\infty \norm{\bdone-A_n} + \sum_{n=1}^\infty \norm{\widehat{A}_n-\bdone} <\infty.
\end{aligned}
\end{multline*}
This shows that $\sigma_n$ is Cauchy, and so converges.
\end{proof}
\end{section}

\begin{example}
Let $D_k=\left( \begin{smallmatrix} (k+1)/k & 0 \\ 0 & 1 \end{smallmatrix} \right)$ for $k\ge 1$. Note that $D_k\to\bdone$.

Pick some unitary $\tau$, and define the sequence $\widehat{A}_n$ as follows: $\widehat{A}_1=\tau ^* D_1 \tau$, $\widehat{A}_2=D_1$, $\widehat{A}_3=D_1^{-1}$, $\widehat{A}_4=D_2$, $\widehat{A}_5=D_3$, $\widehat{A}_6=D_3^{-1}$, $\widehat{A}_7=D_2^{-1}$, $\widehat{A}_8=D_4$, and so on: we define $\widehat{A}_k$'s for $2^j \le k < 2^{j+1}$ in terms of further and further chunks of sequence $D_k$ as
\begin{align*}
&\widehat{A}_{2^j}=D_{2^{j-1}}, \ldots, \widehat{A}_{3\cdot2^{j-1}-1}=D_{2^{j}-1}, \\
&\widehat{A}_{3\cdot 2^{j-1}}=D_{2^{j}-1}^{-1}, \ldots, \widehat{A}_{2^{j+1}-1}=D_{2^{j-1}}^{-1}.
\end{align*}
Note that $\widehat{A}_n>0$, i.e. the sequence corresponds to a block Jacobi matrix of type $1$. Using the notation from Section 2, let $\widehat{Q}_n=\widehat{A}_1\ldots \widehat{A}_n$. Then
\begin{equation*}
\widehat{Q}_{2^j-1}=\widehat{A}_1, \quad \widehat{Q}_{3 \cdot 2^{j-1}-1}=\widehat{A}_1 D_{2^{j-1}} \ldots D_{2^{j}-1}=\widehat{A}_1 D_1,
\end{equation*}
and $\sigma_{2^j}=\phi(\widehat{Q}_{2^j-1})^*=\bdone$, $\sigma_{3 \cdot 2^{j-1}}=\phi(\widehat{Q}_{3 \cdot 2^{j-1}-1})^*=\phi(\tau ^* D_1 \tau D_1)^*$. Now choose $\tau$ such that $\phi(\tau ^* D_1 \tau D_1)$ is not positive definite. This gives that $\lim_{n\to\infty}\sigma_n$ doesn't exist, i.e. type $1$ and type $2$ are not asymptotic to each other.

Of course, the reason is that \eqref{eq2.6} fails here: $\sum \norm{\bdone-A_n A_n^*}$ diverges as $\sum \frac1n$.
\end{example}

{\appendix

\begin{section}{Proof of Li's lemma}
\begin{proof}[Proof of Lemma \ref{li}]
Let $B=U \Sigma V^*$ and $BD=\widetilde{U} \widetilde{\Sigma} \widetilde{V}^*$ be the singular value decompositions of $B$ and $BD$ (i.e. $U, \widetilde{U}, V, \widetilde{V}$ are unitary, and $\Sigma, \widetilde{\Sigma}$ are positive and diagonal). Denote
\begin{equation*}
\begin{aligned}
Y&=\widetilde{U}^*(B-BD)V=\widetilde{U}^*U\Sigma-\widetilde{\Sigma}\widetilde{V}^*V, \\
Z&=U^*(B-BD)\widetilde{V}=\Sigma V^*\widetilde{V} -U^*\widetilde{U} \widetilde{\Sigma}.
\end{aligned}
\end{equation*}
Then
\begin{equation}\label{a2}
Y-Z^*=(\widetilde{U}^* U-\widetilde{V}^* V)\Sigma+\widetilde{\Sigma}(\widetilde{U}^* U-\widetilde{V}^* V) = X\Sigma+\widetilde{\Sigma}X,
\end{equation}
where $X=\widetilde{U}^* U-\widetilde{V}^* V$. On the other hand,
\begin{equation}\label{a3}
\begin{aligned}
Y-Z^*&=\widetilde{U}^*(B-BD)V-\widetilde{V}^*(B^*-D^*B^*)U \\
&= \widetilde{\Sigma} \widetilde{V}^* (D^{-1}-\bdone) V -\widetilde{V}^*(\bdone-D^*)V \Sigma =
\widetilde{\Sigma} E -F \Sigma,
\end{aligned}
\end{equation}
where $E=\widetilde{V}^* (D^{-1}-\bdone) V$, $F=\widetilde{V}^*(\bdone-D^*)V$. Note that $\Sigma$ and $\widetilde{\Sigma}$ are diagonal, and therefore, the solution of \eqref{a2}=\eqref{a3} is
\begin{equation*}
x_{ij}= \frac{\widetilde{\sigma}_{ii}e_{ij}-f_{ij}\sigma_{jj}}{\sigma_{jj}+\widetilde{\sigma}_{ii}},
\end{equation*}
where $X\equiv(x_{ij})$, $E\equiv(e_{ij})$, $F\equiv(f_{ij})$, $\Sigma\equiv(\sigma_{ij})$, $\widetilde{\Sigma}\equiv(\widetilde{\sigma}_{ij})$. Note that $\sigma_{jj}>0$ and $\widetilde{\sigma}_{ii}>0$, and thus by the Schwarz inequality,
\begin{equation*}
|x_{ij}|^2\le \frac{\sigma_{jj}^2+\widetilde{\sigma}_{ii}^2}{(\sigma_{jj}+\widetilde{\sigma}_{ii})^2} (|e_{ij}|^2+|f_{ij}|^2)\le |e_{ij}|^2+|f_{ij}|^2,
\end{equation*}
which implies
\begin{equation*}
||X||^2_{HS}\le ||E||^2_{HS}+||F||^2_{HS}=||\bdone-D^{-1}||^2_{HS}+||\bdone-D||^2_{HS}.
\end{equation*}
Finally, note that $\phi(B)=UV^*$ and $\phi(BD)=\widetilde{U}\widetilde{V}^*$, so $||\phi(B)-\phi(BD)||_{HS}=||\widetilde{U} X V^*||_{HS}=||X||_{HS}$, and we are done.
\end{proof}
\end{section}


\bibliographystyle{amsplain}

\begin{thebibliography}{7}

\bibitem{AN} A. Aptekarev, E. Nikishin, The scattering problem for a discrete Sturm--Liouville operator, Mat. Sb. 121 (163) (1983), 327--358.

\bibitem{Bhatia} R. Bhatia, Matrix Analysis, Springer-Verlag, New York, 1997.

\bibitem{DPS} D. Damanik, A. Pushnitski, B. Simon, The analytic theory of matrix
orthogonal polynomials, Surveys in Approximation Theory 4 (2008), 1--85.

\bibitem{Kozhan} R. Kozhan, Szeg\H{o} asymptotics for matrix-valued measures
with countably many bound states, to appear in J. Approx. Theory.

\bibitem{Kozhan2} R. Kozhan, Jost function for matrix orthogonal polynomials, preprint.

\bibitem{Li} R.-C. Li, Relative perturbation bounds for the unitary polar factor,  BIT  37  (1997),  no. 1, 67--75.

\bibitem{Weyl} H. Weyl, Inequalities between the two kinds of eigenvalues of a linear transformation,  Proc. Natl. Acad. Sci. USA 35 (7) (1949), 408--411.

\end{thebibliography}

\end{document}